\documentclass[a4paper,reqno,11pt]{amsart}
\usepackage[T1]{fontenc}
\usepackage[utf8]{inputenc}
\usepackage{amssymb, amsmath, amsthm, graphicx, color}
\usepackage[inline]{enumitem}
\usepackage[usenames,dvipsnames,svgnames,table]{xcolor}
\usepackage[foot]{amsaddr}
\usepackage{comment}

\usepackage{booktabs}
\usepackage{multirow}
\usepackage[normalem]{ulem}

\usepackage[longnamesfirst,numbers,sort&compress]{natbib}

\usepackage{comment}

\makeatletter
\def\thm@space@setup{
\thm@preskip=4mm
\thm@postskip=0mm
}
\makeatother

\usepackage{hyperref}
\hypersetup{
    colorlinks,
    linkcolor={RoyalBlue},
    citecolor={RubineRed},
    urlcolor={blue!80!black}
}

\usepackage{cleveref}

\usepackage{mathtools}

\usepackage{thmtools, thm-restate}

\setenumerate{label=\textup{(\roman*)}, noitemsep, topsep=0pt,
labelindent=.2em, leftmargin=*, widest=iii,}
\setitemize{noitemsep, topsep=-\parskip, labelindent=.2em, leftmargin=*, widest=iii,}

\newdimen\mywidth
\sbox0{m}%
\newdimen\mywidthprim
\sbox0{m}%
\newdimen\mywidthprimprim
\sbox0{m}%
\newdimen\mywidthA
\sbox0{m}%
\newdimen\mywidthAprim
\sbox0{m}%
\theoremstyle{plain}
\newtheorem{thm}{Theorem}
\newtheorem*{thm*}{Theorem}

\newtheorem*{lemma*}{Lemma}

\newtheorem*{cor*}{Corollary}

\newtheorem*{corollary*}{Corollary}

\theoremstyle{remark}

\crefname{obs}{Observation}{Observations}
\theoremstyle{definition}

\newtheorem*{conj*}{Conjecture}
\crefname{lem}{Lemma}{Lemmas}
\crefname{thm}{Theorem}{Theorems}
\crefname{cor}{Corollary}{Corollaries}

\newcommand{\Oh}{\mathcal{O}}

\newcommand{\polylog}{\operatorname{polylog}}

\let\leq\leqslant
\let\geq\geqslant

\let\epsilon\varepsilon

\title{Note on the treewidth of graphs excluding a disjoint union of cycles as a minor}

\begin{document}

\author[Joret]{Gwena\"el Joret}
\address[G.~Joret]{D\'epartement d'Informatique, Universit\'e libre de Bruxelles, Belgium}
\email{gwenael.joret@ulb.be}

\author[Micek]{Piotr Micek}
\address[P.~Micek]{Department of Theoretical Computer Science, Jagiellonian University, Kraków, Poland}
\email{piotr.micek@uj.edu.pl}

\thanks{G.\ Joret is supported by the Belgian National Fund for Scientific Research (FNRS). P.\ Micek is supported by the National Science Center of Poland under grant UMO-2023/05/Y/ST6/00079 within the WEAVE-UNISONO program}

\begin{abstract}
For a planar graph $H$, let $f(H)$ denote the minimum integer such that all graphs excluding $H$ as a minor have treewidth at most $f(H)$. 
We show that if $H$ is a disjoint union of $k$ cycles then $f(H)=\Oh(|V(H)| + k \log k)$, which is best possible.  
\end{abstract}

\maketitle

\section{Introduction}

By the celebrated \emph{Grid Minor Theorem} of Robertson and Seymour~\cite{RS1986}, every graph that excludes a fixed planar graph $H$ as a minor has treewidth bounded from above by some function of the number of vertices of $H$.  
For every graph $H$, 
let $f(H)$ be the least nonnegative integer such that 
graphs excluding $H$ as a minor have treewidth at most $f(H)$. 
The current best upper bound for $f(H)$ when $H$ is planar is $\Oh(|V(H)|^9\polylog|V(H)|)$, due to~\citet{CT21}.

Better bounds on $f(H)$ are known in some special cases.  
In particular, the case where $H$ consists of a collection of disjoint cycles is well studied. 
If $H$ is just one cycle, then a theorem of \citet{B03} shows that $f(H) \leq |V(H)|-2$. 
If $H$ is the disjoint union of $k$ triangles, then the classic theorem of~\citet{EP1965} implies that $f(H) \in \Oh(k \log k)$, which is tight.  
More generally, if $H$ is the disjoint union of $k$ cycles of length $\ell$, then $f(H) \in \Oh(k\ell + k \log k)$ as shown by~\citet{MNSW17}, which is again tight.  

In this note, we focus on the case where $H$ is the disjoint union of $k$ cycles, with no constraint on their lengths. 
The study of this case was initiated recently by~\citet{GHOR24}, who proved that  
\[
f(H) \leq \frac{3}{2}|V(H)| +  \Oh(k^2 \log k),  
\]
and by~\citet{HLRW25}, who proved that  
\[
f(H) \in \Oh(|V(H)| \log (k+1) + k \log k \log \ell)
\] 
where $\ell$ denotes the maximum length of a cycle in $H$.  

The following lower bound on $f(H)$ can be shown to hold when $H$ is the disjoint union of $k$ cycles: 
\begin{equation*}
    f(H) \in \Omega(|V(H)| + k \log k).
\end{equation*} 
Indeed, $f(H) \geq |V(H)|-2$ holds since the complete graph on $|V(H)|-1$ vertices has no $H$ minor and has treewidth $|V(H)|-2$. 
Also, $f(H) \in \Omega(k \log k)$ holds because of the existence of $n$-vertex cubic graphs with treewidth $\Omega(n)$ and girth $\Omega(\log n)$ (see e.g.
~\cite{M94}): For an adequate choice of $n\in \Theta(k \log k)$, this gives graphs with treewidth $\Theta(k \log k)$ containing no $k$ vertex-disjoint cycles, and thus having no $H$ minor. 

Our main result is the following upper bound on $f(H)$ that matches the above lower bound up to a constant factor. 
(All logarithms in this paper are in base $2$.)

\begin{thm}
\label{thm:main}
If $H$ is a graph consisting of $k$ vertex-disjoint cycles, then 
\[
f(H) \leq 6|V(H)| + 10k\log k + 10 k \log \log k + 40k.
\]
\end{thm}

In both papers~\cite{GHOR24} and~\cite{HLRW25}, the proofs exploit brambles, which are dual objects to tree decompositions, to construct an $H$ minor in case the treewidth is too large.  
In our proof, we follow a different approach; namely, we let a slight modification of the proof of a theorem about the so-called Erdős--Pósa property of long cycles due to \citet{MNSW17} do all the heavy lifting, and we conclude using a short inductive argument, relying on the bound on the treewidth of graphs with no long cycles due to \citet{B03}.

\section{Proof}
The starting point of our proof is the following result of \citet{MNSW17}.

\begin{thm}[\citet{MNSW17}]
\label{thm:EP_long_cycles}
For every integer $\ell \geq 3$ and every integer $k\geq 1$, every graph $G$ either has $k$ vertex-disjoint cycles each of length at least $\ell$, or a vertex subset $X$ of size 
\[
|X| \leq \left\{
\begin{array}{ll}
  6k\ell + 10 k\log k + 10 k\log \log k + 40k   &  \textrm{ if } k\geq 2\\
  0 & \textrm{ if } k = 1
\end{array}
\right.
\]
meeting all cycles of length at least $\ell$. 
\end{thm}

A close inspection of the proof of \cref{thm:EP_long_cycles} in \cite{MNSW17} reveals that the $6k\ell$ term in the bound is only due to the potential existence of cycles in $G$ whose lengths are in the interval $[\ell, 6\ell]$. 
In this case, the proof proceeds by removing the vertex set of such a cycle from $G$ and applying induction on the resulting graph, incurring a cost of at most $6\ell$ per cycle removed. 
For our purposes, the key observation is that, if the graph $G$ under consideration has no cycle of length between $\ell$ and $6\ell$, then the extra $6k\ell$ term in the bound on the size of $X$ can be dropped: 

\begin{thm}[\citet{MNSW17}, implicit]
\label{thm:EP_long_cycles_without_medium_length_cycles}
For every integer $\ell \geq 3$ and every integer $k\geq 1$, every graph $G$ 
with no cycle of length at least $\ell$ and at most $6\ell$
either has $k$ vertex-disjoint cycles each of length at least $\ell$, or a vertex subset $X$ of size 
\[
|X| \leq \left\{
\begin{array}{ll}
  10 k\log k + 10 k\log \log k + 40k   &  \textrm{ if } k\geq 2 \\
  0 & \textrm{ if } k = 1
\end{array}
\right.
\]
meeting all cycles of length at least $\ell$. 
\end{thm}

We will also use the following result of \citet{B03} already mentioned in the introduction.

\begin{thm}[\citet{B03}]
\label{thm:Birmele} 
For every integer $\ell \geq 3$, every graph with no cycle of length at least $\ell$ has treewidth at most $\ell-2$. 
\end{thm}

We may now turn to the proof of our main theorem. 

\begin{proof}[Proof of \cref{thm:main}]
For all integers $h \geq 0$ and $k\geq 1$, let 
\[
g(h, k) := 6h + 10 k\log k + 10 k \log \log k + 40k.
\] 
Let $H$ be an $h$-vertex graph that is the disjoint union of $k$ cycles. 
Let $G$ be a graph with no $H$ minor. 
Our goal is to show that $G$ has treewidth at most $g(h, k)$. 
The proof is by induction on $|V(G)|$. 

Let $\ell$ be the length of a longest cycle in $H$.
If $k=1$, then $G$ has no cycle of length at least $\ell$, and thus by \cref{thm:Birmele} $G$ has treewidth at most
$\ell-2 \leq h \leq g(h, 1)$, as desired. 
Thus, we may assume that $k\geq 2$. 

If $G$ contains a cycle $C$ of length at least $\ell$ and at most $6\ell$, we proceed as follows. 
Let $H'$ be the $h'$-vertex graph obtained from $H$ by removing the vertices of a cycle of length $\ell$.  
Observe that $H'$ is non empty since it consists of $k-1\geq 1$ cycles, and that $G-V(C)$ has no $H'$ minor. 
By induction, $G-V(C)$ has treewidth at most $g(h', k-1)$.    
Adding all the vertices of $C$ in all the bags of a minimum-width tree decomposition of $G-V(C)$, we deduce that $G$ has treewidth at most 
\[
g(h', k-1) + 6\ell 
= g(h', k-1) + 6(h-h')
\leq g(h,k),
\]
as desired. 
Hence, we may assume that no cycle in $G$ has length between $\ell$ and $6\ell$. 

Let $r$ be the maximum number of vertex-disjoint cycles of length at least $\ell$ in $G$. 
Observe that $r \leq k - 1$, since otherwise $G$ would have $H$ as a minor.  
Applying \cref{thm:EP_long_cycles_without_medium_length_cycles} on $G$, we obtain a set $X$ of vertices of $G$ that meets all cycles of length at $\ell$ in $G$ and of size $|X| \leq g(0, r+1)$. 
Since $G-X$ has no cycle of length at least $\ell$, its treewidth is at most $\ell - 2$ by \cref{thm:Birmele}.  
Adding all the vertices of $X$ in all the bags of a minimum-width tree decomposition of $G-X$, we deduce that $G$ has treewidth at most 
\[
g(0, r+1) + (\ell-2)
\leq g(0, k) + h  
\leq g(h, k) 
\]
as desired. 
\end{proof}

\section*{Acknowledgments}

This research was carried out during the second edition of the Structural Graph
Theory Workshop (STWOR) held at the conference center of the University of Warsaw in Chęciny (Poland), June 30th--July 5th, 2024. 
We thank the workshop organizers and other participants for providing a stimulating working environment. 
We also thank the authors of~\cite{HLRW25} for helpful discussions. 

\bibliographystyle{abbrvnat}
\bibliography{bibliography}

\end{document}